\documentclass{amsart}
\usepackage[usenames,dvipsnames]{xcolor}
\usepackage{amsmath}
\usepackage{amsthm}
\usepackage{amssymb}
\usepackage{tikz-cd}
\usepackage{etoolbox}
\usepackage{enumitem}

\usepackage{tikz}
\usetikzlibrary{calc}

\newtheorem{thm}{Theorem}[section]
\newtheorem{theorem}[thm]{Theorem}
\newtheorem{corollary}[thm]{Corollary}

\newtheorem{lemma}[thm]{Lemma}

\newtheorem{proposition}[thm]{Proposition}

\theoremstyle{definition}
\newtheorem{definition}[thm]{Definition}

\newtheorem{remark}[thm]{Remark}

\newcommand{\gufin}{(G/U)(\mathbb{F}_{q^k})}

\begin{document}

\title[Symplectic Fourier transforms and the algebra of braids and ties]{Symplectic Fourier-Deligne transforms on $G/U$ and the algebra of braids and ties}
\author{Calder Morton-Ferguson}

\begin{abstract}
We explicitly identify the algebra generated by symplectic Fourier-Deligne transforms (i.e. convolution with Kazhdan-Laumon sheaves) acting on the Grothendieck group of perverse sheaves on the basic affine space $G/U$, answering a question originally raised by A. Polishchuk. We show it is isomorphic to a distinguished subalgebra, studied by I. Marin, of the generalized algebra of braids and ties (defined in Type $A$ by F. Aicardi and J. Juyumaya and generalized to all types by Marin), providing a connection between geometric representation theory and an algebra defined in the context of knot theory. Our geometric interpretation of this algebra entails some algebraic consequences: we obtain a short and type-independent geometric proof of the braid relations for Juyumaya's generators of the Yokonuma-Hecke algebra (previously proved case-by-case in types $A, D, E$ by Juyumaya and separately for types $B, C, F_4, G_2$ by Juyumaya and S. S. Kannan), a natural candidate for an analogue of a Kazhdan-Lusztig basis, and finally an explicit formula for the dimension of Marin's algebra in Type $A_n$ (previously only known for $n \leq 4$).
\end{abstract}

\maketitle

\setcounter{tocdepth}{1}
\tableofcontents

\section{Introduction}

In \cite{KL}, given a semisimple simply-connected algebraic group $G$ split over a finite field $\mathbb{F}_{q}$ with associated Weyl group $W$, the authors define an action of the generalized braid group $B_W$ on the constructible derived category $D^b_{c}(G/U)$ of mixed $\ell$-adic sheaves on the basic affine space associated to $G$ via so-called ``symplectic Fourier-Deligne transforms." In \cite{P}, this action was studied in further detail, where Polishchuk showed that each of the standard braid generators for $B_W$ satisfies a cubic relation when acting on $K_0(G/U) = K_0(D^b_c(G/U))$. In doing so, Polishchuk shows that the action of $B_W$ factors through the ``cubic Hecke algebra" $\mathcal{H}_W^c$ given by the quotient of $\mathbb{C}(v)[B_W]$ by a certain cubic relation in each simple braid generator. In loc. cit., Polishchuk asks which algebra is really acting: are there other relations satisfied by the action of $\mathbb{C}(v)[B_W]$ on $K_0(G/U)$, and can we identify the quotient of $\mathcal{H}^c_W$ by these relations more explicitly? The central result of this paper is an answer to this question.

Another perspective on symplectic Fourier-Deligne transforms can be given via convolution with ``Kazhdan-Laumon sheaves" on $G/U \times G/U$ as follows. In \cite{KL}, Kazhdan and Laumon define for each $w \in W$ a stratum $X(w) \subset G/U \times G/U$ along with a perverse sheaf $K(w)$ on $X(w)$ and its Goresky-MacPherson extension $\overline{K(w)}$ to the closure $\overline{X(w)}$. They then show that the action of the symplectic Fourier-Deligne transform on $D^b_c(G/U)$ corresponding to any simple reflection $s$ is given by convolution on the right with $\overline{K(s)}$; further, each of these sheaves $\overline{K(w)}$ is a pullback under a multiplication map $G/U \times G/U \to U\backslash G/U$, with their convolution agreeing with a suitably-defined convolution of sheaves on $U\backslash G/U$. As a result, the subalgebra of endomorphisms of $K_0(G/U)$ generated by symplectic Fourier-Deligne transforms is isomorphic to the algebra $\mathrm{KL}(v) \subset K_0(U \backslash G/U)$ generated by classes of Kazhdan-Laumon sheaves under convolution, and this is the perspective we will take on this algebra for the purposes of this paper.

In \cite{AJ}, the authors define a certain diagrammatic algebra called the algebra of braids and ties (or sometimes the ``bt-algebra") over the field $\mathbb{C}(v)$, generalizing certain properties of the Yokonuma-Hecke algebra of Type $A_n$ to the case of a generic parameter. In \cite{MarYoko}, the author constructs a generalization (called $C_W^R$ in loc. cit.) of this algebra to arbitrary Coxeter type, which we call the generalized algebra of braids and ties and denote by $\mathcal{E}(v)$. In loc. cit., the author also considers a natural subalgebra which we call $\mathcal{C}(v) \subset \mathcal{E}(v)$ which is the image of a natural homomorphism $\mathbb{C}(v)[B_W] \to \mathcal{E}(v)$. In the next result, we show that Kazhdan-Laumon sheaves under convolution categorify $\mathcal{C}(v)$, answering the aforementioned question asked in \cite{P}. Further, we show that the whole algebra $\mathcal{E}(v)$ can be obtained by the convolution algebra generated by a modest enlargement of this Kazhdan-Laumon algebra, i.e. by Kazhdan-Laumon sheaves in addition to certain other sheaves $E(s)$ introduced in Definition \ref{def:es} for each simple reflection $s$.

\begin{theorem}\label{thm:mainthm}
Let $\mathrm{KL}(v)$ be the algebra sitting inside $K_0(G/U)\otimes \mathbb{C}(v)$ generated by the classes of Kazhdan-Laumon sheaves $\overline{K(w)}$ for all $w \in W$ under convolution. Let $\widetilde{\mathrm{KL}}(v)$ be the algebra generated by the classes of Kazhdan-Laumon sheaves in addition to the sheaves $E(s)$ for all $s \in S$. 
\begin{enumerate}[label=(\alph*)]
    \item The algebra $\widetilde{\mathrm{KL}}(v)$ is isomorphic to the generalized algebra of braids and ties $\mathcal{E}(v)$, which has generators and relations given in Definition \ref{def:evcv}.
    \item The algebra $\mathrm{KL}(v)$ is isomorphic to the distinguished subalgebra $\mathcal{C}(v)$ of $\mathcal{E}(v)$ generated only by the generators $g_s$, $s \in S$.
\end{enumerate}
\end{theorem}
This upgrades Polishchuk's observation that $\mathrm{KL}(v)$ is a quotient of $\mathcal{H}^c_W$ and provides an explicit description of this algebra.

The generalized braids and ties algebra $\mathcal{E}(v)$ was defined as a generic-parameter version of the Yokonuma-Hecke algebra $\mathrm{Y}_{q,k}$ over $\mathbb{F}_{q^k}$, as defined in \cite{Yoko}. In \cite{Juyu} and \cite{JuyKan}, the authors gave a new presentation of $\mathrm{Y}_{q,k}$ which has since been used fruitfully in the study of this algebra and the algebra of braids and ties (e.g. in \cite{SRH} and \cite{SRH2}). This presentation uses a generating set $\{L_s\}_{s \in S}\cup \{R_t\}_{t \in T(\mathbb{F}_{q^k})}$ with relations given in Proposition \ref{prop:yokojuyu}.

We note that in \cite{Juyu} and \cite{JuyKan}, to complete their proof of the relations satisfied by their new generators for $\mathrm{Y}_{q,k}$, the authors prove case-by-case that the generators $L_s$ satisfy braid relations, proving this for simply-laced types in \cite{Juyu} and then for all other types in \cite{JuyKan}. On the other hand, combining our Theorem \ref{thm:mainthm} with a result from \cite{KL} gives a type-independent proof of this fact using the geometry of perverse sheaves on $G/U\times G/U$, avoiding the algebraic calculations done in loc. cit.

\begin{corollary}[First proved in \cite{Juyu}, \cite{JuyKan}]
The generators $\{L_s\}_{s \in S}$ of $\mathrm{Y}_{q,k}$ introduced by Juyumaya in \cite{Juyu} satisfy braid relations.
\end{corollary}

The following is a diagram summarizing the relationship between the algebras discussed so far, where the two vertical isomorphisms are from Theorem \ref{thm:mainthm}, the arrow $\mathcal{E}(v) \to \mathrm{Y}_{q, k}$ is the specialization map at $v = q^{k/2}$, and the rightmost vertical arrow is the Frobenius trace map $\mathrm{tr}_{q^k}$ which we will describe in Section \ref{subsec:shfun}.
\[
    \begin{tikzcd}
        \mathrm{KL}(v) \arrow[r, hook]\arrow[d, "\sim" style={anchor=north, rotate=90, xshift=1pt}] & \widetilde{\mathrm{KL}}(v) \arrow[r, hook]\arrow[d, "\sim" style={anchor=north, rotate=90, xshift=1pt}] & K_0(G/U) \arrow[dd, two heads]\\
        \mathcal{C}(v) \arrow[r, hook] & \mathcal{E}(v) \arrow[d, two heads]& \\
        & \mathrm{Y}_{q,k} \arrow[r, hook] & \mathbb{C}[(G/U)(\mathbb{F}_{q^k})]
    \end{tikzcd}
\]

We conclude by explaining two more algebraic consequences, both of which arise from questions posed in \cite{MarYoko}. One of the motivating facts about the algebra $\mathcal{C}(v)$ is that it admits a natural quotient map to the Hecke algebra $\mathcal{H}$ by sending $g_s \mapsto -A_s$, where $A_s$ is the standard generator of the Hecke algebra associated to $s \in S$. In the next result, we show that actually this algebra admits a morphism to each monodromic Hecke algebra $\mathbf{H}_{\mathfrak{o}}$ associated to a $W$-orbit $\mathfrak{o}$ of character sheaves on $T$ as described in \cite{LY}.
\begin{theorem}
Let $\mathbf{H}_{\mathfrak{o}}$ be the monodromic Hecke algebra associated to an orbit $\mathfrak{o}$.
\begin{enumerate}[label=(\alph*)]
    \item For any character sheaf $\mathcal{L}$, convolution with $[\Delta(e)_{\mathcal{L}}]$ induces a morphism $\pi_{\mathcal{L}} : \mathrm{KL}(v) \to \mathbf{H}_{\mathfrak{o}}$. In particular, $\pi = \pi_{{\mathcal{L}}_{\mathrm{triv}}}$ is the quotient map to the usual Hecke algebra.
    \item There exists a natural lift $\mathsf{c}_{w}$ in $\mathrm{KL}(v)$ of the Kazhdan-Lusztig basis $C_w$ under the map $\pi$ which satisfies a self-duality condition with respect to a natural involution on $\mathrm{KL}(v)$.
\end{enumerate}
    
\end{theorem}

Finally, in \cite{MarYoko}, the author studies the algebra $\mathcal{C}(v)$ more explicitly in Type $A$ and discusses some connections to the Links-Gould invariant and its defining algebra which was studied in \cite{LinksGould}. In Section 4.3 of \cite{MarYoko}, Marin computes that $\dim \mathcal{C}_{A_1} = 3$, $\dim \mathcal{C}_{A_2} = 20$, and provides some computational evidence that $\dim \mathcal{C}_{A_3}(v) = 217$, $\dim \mathcal{C}_{A_4}(v) = 3364$, noting that this sequence does not exist in the OEIS and asking whether a general formula exists for the dimension. We conclude our paper by confirming Marin's computational results and extending them by providing a general formula for the dimension of $\mathcal{C}_{A_n}(v)$ (and therefore also of the symplectic Fourier transform algebra $\mathrm{KL}(v)$; note that even its finite-dimensionality is not obvious a priori without Theorem \ref{thm:mainthm}) in Type $A_n$ for all $n \geq 1$.

\begin{theorem}
When considered as an algebra over $\mathbb{C}(v)$, there is an explicit formula for the dimension of the algebra $\mathcal{C}(v)$ given by
\begin{align}
    \dim \mathcal{C}(v) = |W|\sum_{I \subset S} D_I/N_I
\end{align}
where $N_I$ is the size of the normalizer of the standard parabolic subgroup of $W$ associated to $I \subset S$, and $D_I$ is the number of elements of $w \in W$ such that for each contiguous block $I' \subset I$, there is a simple reflection $s \in I'$ for which $\ell(ws) < \ell(w)$.

This gives a general formula which continues the sequence of dimensions
\begin{align*}
    1, ~ 3, ~ 20, ~ 217, ~ 3364, \dots
\end{align*}
given in \cite{MarYoko}.
\end{theorem}

In conclusion, we hope that this paper will serve two purposes. First, it serves as an explicit algebraic identification of the algebra generated by symplectic Fourier-Deligne transforms in terms of a known finite-dimensional algebra -- a project begun by Polishchuk's identification of the cubic relations in $\mathrm{KL}(v)$ more than two decades ago. In future work, we hope to exploit this identification to prove some of the conjectures about representations of the braid group factoring through the symplectic Fourier transform algebra which were originally posed in \cite{P}. Secondly, our main result gives a geometric interpretation for some algebras familiar to experts in the theory of braids and knots with potential for applications to some purely algebraic questions. This paper contains some basic examples of such algebraic questions which we answer via this geometric perspective, and we hope that further study of this algebra will be enhanced by this interpretation which links the algebraic perspective to the rich set of tools which become available in the context of perverse sheaves and algebraic geometry.


\subsection{Acknowledgments} I would like to thank Ivan Marin and Louis Funar for pointing me in the right direction in the early stages of my attempt to explicitly identify the algebra generated by symplectic Fourier transforms. I would also like to acknowledge Minh-Tam Trinh and Elijah Bodish for their mentorship and many inspiring discussions which played a key role in guiding the directions pursued in this paper, and for helpful comments on a draft. I would also like to thank my advisor, Roman Bezrukavnikov, for introducing me to Kazhdan-Laumon's construction and for his feedback and support. I gained much from conversations with Alexander Polishchuk and Ben Elias during the writing of this paper. I would also like to thank the referees, one of whose suggestions provided simpler and more conceptual proofs of Propositions \ref{prop:yoko} and \ref{prop:monodrom}. Finally, I acknowledge the support of the Natural Sciences and Engineering Research Council of Canada (NSERC).

\section{Preliminaries}

Let $q$ be a power of some odd prime $p$, let $k = \mathbb{F}_{q}$, and let $G$ be a split semisimple simply-connected algebraic group over $k$. Let $T$ be a Cartan subgroup split over $k$, $B$ a Borel subgroup containing $T$, and $U$ its unipotent radical. Let $X = G/U$ be the basic affine space associated to $G$ considered as a variety over $k$. Let $W$ be the Weyl group, with $S \subset W$ a set of simple reflections. We let $\Phi$ be the set of roots, and we fix a choice of identification of simple reflections $s \in S$ with simple roots $\alpha_s \in S$. For any two simple reflections $s, s' \in S$, we let $m_{ss'}$ be the order of $ss'$ in $W$.

For every simple root $\alpha_s$ we fix an isomorphism of the corresponding one-parameter subgroup $U_s \subset U$ with the additive group $\mathbb{G}_{a,k}$. This uniquely defines a homomorphism $\rho_s : SL_{2,k} \to G$ which induces the given isomorphism of $\mathbb{G}_{a,k}$ (embedded in $SL_{2,k}$ as upper-triangular matrices) with $U_s$; then let
\begin{align*}
    n_s & = \rho_s
    \begin{pmatrix}
        0 & 1\\-1 & 0 
    \end{pmatrix}.
\end{align*}
For any $w \in W$, writing a reduced word $w = s_{i_1} \dots s_{i_k}$ we set $n_w = n_{s_{i_1}} \dots n_{s_{i_k}}$, and one can check that this does not depend on the reduced word. We also define for any $s \in S$ the subtorus $T_s \subset T$ obtained from the image of the coroot $\alpha_s^\vee$ and define $T_w$ for any $w \in W$ to be the product of all $T_s$ ($s \in S$) for which $s \leq w$ in the Bruhat order.

Once and for all, we pick a square root $q^{\frac{1}{2}}$ of $q$, and an additive character $\psi : \mathbb{F}_p \to \mathbb{C}^\times$. Note that this induces natural additive characters of $\mathbb{F}_{q}$ for all powers $q$ of $p$ by composing with the field trace; we will also denote these characters by $\psi$. Let $\mathcal{L}_\psi$ be the associated Artin-Schreier sheaf on $\mathbb{G}_{a, \mathbb{F}_q}$.

We let $K_0(G/U)$ denote the Grothendieck group of the constructible derived category $D^b_c(G/U)$ of mixed $\ell$-adic sheaves on $G/U$. The ``half Tate twist" $(\tfrac{1}{2})$ corresponding to $q^{\frac{1}{2}}$ gives $K_0(G/U)$ the natural structure of a $\mathbb{Z}[v, v^{-1}]$-module. We choose now an isomorphism $\overline{\mathbb{Q}}_{\ell} \cong \mathbb{C}$ and choose to work only with the field $\mathbb{C}$ from now on; in doing so, we will then be able to make reference to complex conjugation. By abuse of notation, we will write simply $K_0(G/U)$ when referring to $K_0(G/U) \otimes_{\mathbb{Z}[v, v^{-1}]} \mathbb{C}(v)$ going forward.

Later we will consider character sheaves $\mathcal{L}$ on the torus, and relatedly multiplicative characters $\theta : T(\mathbb{F}_{q^k}) \to \mathbb{C}^\times$. To either of these objects we can associate a reflection subgroup $W_{\mathcal{L}}^\circ$ or $W_{\theta}^\circ$ of $W$ generated by reflections corresponding to coroots on which $\mathcal{L}$ or $\theta$ is trivial; c.f. \cite{LY} for more details.

Let $\mathcal{H}$ be the usual Hecke algebra over $\mathbb{Z}[v, v^{-1}]$ (which we sometimes consider over $\mathbb{C}(v)$) generated by $\{A_w\}_{w \in W}$, normalized so that the quadratic relation satisfied by any $A_s$ for $s \in S$ takes the form
\begin{align}
A_s^2 & = (v^2 - 1)A_s + v^2.
\end{align}

For any $w \in W$, let $\tilde{A}_w = v^{-\ell(w)}A_w$, so that for any $s \in S$,
\begin{align}\label{eqn:heckerelation}
(\tilde{A}_s - v)(\tilde{A}_s + v^{-1}) & = 0.
\end{align}

There is a natural involution on $\mathcal{H}$ given by $\overline{A_{w}} \mapsto A_{w^{-1}}^{-1}$ and $\overline{v} = v^{-1}$.
Following the conventions of Section 3.5 of \cite{W}, recall the canonical basis $C_w$ of $\mathcal{H}$ given in terms of the Kazhdan-Lusztig polynomials $P_{x,w}$ by
\begin{align*}
C_w & = \sum_{x \leq w} (-1)^{\ell(w) - \ell(x)} v^{\ell(x)-\ell(w)}P_{x,w} \tilde{A}_{x^{-1}}^{-1}
\end{align*}
which satisfy $\overline{C_w} = C_w$.

\subsection{Kazhdan-Laumon sheaves and convolution}\label{sec:klshconv}
In \cite{KL} and \cite{P}, the authors associate to each $w \in W$ an element of $D^b_{c}(G/U\times G/U)$ which is perverse up to shift and irreducible. 

Following \cite{P}, let $X(w) \subset G/U \times G/U$ be the subvariety of pairs $(gU, g'U) \subset (G/U)^2$ such that $g^{-1}g' \in Un_wT_wU$. There is a canonical projection $\mathrm{pr}_w : X(w) \to T_w$ sending $(gU, g'U)$ to the unique $t_w \in T_w$ such that $g^{-1}g' \in Un_wt_wU$. In the case when $w = s \in S$, the morphism $\mathrm{pr}_s : X(s) \to T_s \cong \mathbb{G}_{m, k}$ extends to $\overline{\mathrm{pr}}_s : \overline{X(s)} \to \mathbb{G}_{a,k}$ and we have
\begin{align*}
    \overline{K(s)} = (-\overline{\mathrm{pr}}_s)^* \mathcal{L}_\psi
\end{align*}
and in the case of general $w \in W$
\begin{align}\label{eqn:wfroms}
    \overline{K(w)} & = \overline{K(s_{i_1})} * \dots * \overline{K(s_{i_k})}
\end{align}
where $w = s_{i_1} \dots s_{i_k}$ is a reduced expression. One can take this as the definition of Kazhdan-Laumon sheaves (as it is well-defined by the proposition below and therefore does not depend on the reduced word given the proposition below) or refer to the explicit definition which works for all $w \in W$ at once given in \cite{KL} or \cite{P}.

The following was proved by a geometric argument by Kazhdan and Laumon in their original paper on the subject; it will be the main tool we use to (re)prove the braid relations for Juyumaya's generators of the Yokonuma-Hecke algebra once we show that these generators are categorified by Kazhdan-Laumon sheaves.

\begin{proposition}[\cite{KL}]
    The Kazhdan-Laumon sheaves $\overline{K(s)}$ for $s \in S$ under convolution satisfy the braid relations (up to isomorphism).
\end{proposition}

In \cite{Eteve} it is observed that the sheaves $\overline{K(w)}$ can be obtained by pullback from a canonical multiplication map $m : G/U \times G/U \to U\backslash G/U$, and that this equivalence respects convolution. As a result, from now on we will consider $\overline{K(w)}$ as an element of $D^b_{c}(U \backslash G/U)$ and we will work directly with convolution of sheaves on $U \backslash G/U$ in the sense of Definition 2.3 of loc. cit. 

We also refer to Definition 2.4 of loc. cit. for the action of $D^b_c(U \backslash G /U)$ on $D^b_c(G/U)$ by convolution on the right, which we will use going forward when we consider the defining action of Kazhdan-Laumon sheaves on $D^b_c(G/U)$, which for $\overline{K(s)}$ agrees with the corresponding symplectic Fourier-Deligne transform on $D^b_c(G/U)$ as we explain in \ref{rem:sft}.

Note that for any $g \in G/U$, we obtain an endomorphism of $D^b_{c}(G/U)$ by pullback along the natural action of $g$ on $G/U$ by multiplication; it is shown in \cite{KL} that convolution by $\overline{K(w)}$ commutes with this action.

\begin{remark}
\label{rem:sft}
In \cite{KL} and \cite{P}, the authors study symplectic Fourier-Deligne transforms on the basic affine space corresponding to simple reflections $s \in S$. It is shown in loc. cit. that the symplectic Fourier-Deligne transform of sheaves on $G/U$ associated to $s \in S$ agrees with convolution by the Kazhdan-Laumon sheaf $\overline{K(s)}$, and so the ``symplectic Fourier transform algebra" generated by the induced endomorphisms of $K_0(G/U)$ under composition is the same as the subalgebra of $K_0(G/U)$ generated by Kazhdan-Laumon sheaves under convolution which we study in the present paper.
\end{remark} 

In addition to the Kazhdan-Laumon sheaves $\overline{K(s)}$ introduced above (which in the next section we will show categorify the generators $g_s$ of the generalized algebra and braids and ties which we will introduce later), it will be useful to introduce another collection of sheaves $E(s)$ which will categorify the other generators $e_s$ of this algebra.

\begin{definition}\label{def:es}
    For any $s \in S$, note that there is a natural inclusion $i_{T_s} : T_s \hookrightarrow G/U$ with $i_{T_s}(t_s) = t_sU$. Let $\underline{\mathbb{C}}_{T_s}$ be the constant sheaf on $T_s$ and define the sheaf $E(s) = i_{T_s!}\underline{\mathbb{C}}_{T_s}$.
\end{definition}

\begin{definition}
    Let $\mathrm{KL}(v)$ be the algebra generated by the images in $K_0(G/U)$ of the classes $[\overline{K(s)}]$ of Kazhdan-Laumon sheaves for all $s \in S$ under convolution. Let $\widetilde{\mathrm{KL}}(v)$ be the subalgebra of $K_0(G/U)$ generated by $\mathrm{KL}(v)$ along with the $[E(s)]$ for all $s \in S$.
    
    We use the notation $\mathsf{a}_w = [\overline{K(w)}]$ for any $w \in W$ and $\mathsf{e}_s = [E(s)]$ for any $s \in S$.
\end{definition}

Alternatively, $\mathrm{KL}(v)$ could also be defined as the subalgebra of $K_0(G/U \times G/U)$ generated by the Kazhdan-Laumon sheaves on $G/U \times G/U$ (as originally defined in \cite{KL}) under convolution.

\subsection{The algebras $\mathcal{E}(v)$ and $\mathcal{C}(v)$}

We now recall the definition of the ``generalized algebra of braids and ties." This was defined in Type $A$ by Aicardi and Juyumaya in \cite{AJ}, and subsequently generalized to all types by Marin in \cite{MarYoko}. Here, to unify notation we use $\mathcal{E}(v)$ (originally used in \cite{AJ} to refer only to the original definition in Type $A$) to denote the generalized algebra in the sense of Marin (which is called $C_W^R$ in his paper), whose definition is very similar and agrees in Type $A$ with the original algebra of braids and ties.

\begin{definition}\label{def:evcv}
    The generalized algebra of braids and ties $\mathcal{E}(v)$ is the algebra generated as follows. For each $s \in S$, we have a generator $g_s$, for each $r \in W$ a reflection, we have a generator $e_r$. We begin with the relations.
\begin{align*}
    \underbrace{g_sg_tg_s \dots}_{m_{st}} & = \underbrace{g_tg_sg_t \dots}_{m_{st}} \quad s, t \in S\\
    e_r^2 & = e_r \quad \text{for all reflections $r \in W$}\\
    e_{r_1}e_{r_2} & = e_{r_2}e_{r_1}\\
    e_{r_1}e_{r_2} & = e_{r_1}e_{r_1^{-1}r_2r_1}\quad \text{for all reflections $r_1, r_2 \in W$}\\
    g_se_r & = e_{srs}g_s \quad \text{for all reflections $r$ and all $s \in S$}\\
    g_s^2 & = 1 + (v^2 - 1)e_s(1 + g_s) \quad s \in S
\end{align*}
    Now for any $w \in W$ with reduced expression $s_{i_1} \dots s_{i_k}$, the braid relation above shows that $g_w = g_{s_{i_1}} \dots g_{s_{i_k}}$ is well-defined. Further, for any reflection subgroup $J \subset W$, set $e_J = e_{r_1} \dots e_{r_\ell}$ where $\{r_1, \dots, r_\ell\}$ is a generating set of reflections for $J$. In Section 3.4 of \cite{MarYoko}, a map $p$ from the set of reflection subgroups of $W$ to the set of closed symmetric root subsystems of $\Phi$ is defined. We now impose the additional relation
    \begin{align}
        e_J & = e_{J'}, \quad p(J) = p(J')\label{eqn:CRW}
    \end{align}
    to complete the definition of $\mathcal{E}(v)$. We note that clearly this algebra is also generated only by $\{g_s, e_s\}_{s \in S}$.\footnote{Note that the algebra generated by the first six relations is called $C_W$ in \cite{MarYoko} and our algebra $\mathcal{E}(v)$ obtained by adding in the relation (\ref{eqn:CRW}) is called $C_W^R$ in loc. cit. We also note that (\ref{eqn:CRW}) is redundant in Type $A$, and therefore doesn't occur in the original definition of the algebra of braids and ties.} We then define the algebra $\mathcal{C}(v)$ to be the ``braid image" subalgebra of $\mathcal{E}(v)$, i.e. the subalgebra generated by the $g_s$, $s \in S$.
\end{definition}

\section{Categorification theorems}

\subsection{Sheaves and functions}\label{subsec:shfun}

Since $G/U$ is defined over $\mathbb{F}_q$, any object lying in $D^b_{c}(G/U)$ comes equipped with a natural Frobenius endomorphism. For any power $q^k$ of $q$ for $k \geq 1$, taking the trace of the $k$th power of the Frobenius morphism yields a well-defined map
\begin{align*}
    \mathrm{tr}_{q^k} : K_0(G/U) \to \mathbb{C}[\gufin].
\end{align*}

These assemble to a map
\begin{align*}
    \mathrm{tr} : K_0(G/U) \to \prod_{k \geq 1} \mathbb{C}[\gufin].
\end{align*}

\begin{proposition}[\cite{Lau}]\label{prop:lau}
    $\mathrm{tr}$ is injective.
\end{proposition}

\begin{definition}
For any $k \geq 1$, let $\mathrm{KL}_{q,k}$, an algebra which is a subspace of $\mathbb{C}[\gufin]$, be the image of $\mathrm{KL}(v) \subset K_0(G/U)$ under the morphism $\mathrm{tr}_{q^k}$. Similarly, let $\widetilde{\mathrm{KL}}_{q,k}$ be the image of $\widetilde{\mathrm{KL}}(v)$. We let $\overline{\mathsf{a}_s}$ and $\overline{\mathsf{e}_s}$ denote the images in $\mathrm{KL}_{q,k}$ of the corresponding generators.
\end{definition}

\subsection{Symplectic Fourier transforms on functions}

As mentioned in Remark \ref{rem:sft}, it is shown in \cite{KL} and \cite{P} that convolution with $\overline{K(s)}$ agrees with the symplectic Fourier-Deligne transform defined in loc. cit. In particular, one can then deduce the following description.

Given a simple reflection $s \in S$, let $P_s$ be the corresponding parabolic subgroup of $G$, and let $Q_s = [P_s, P_s]$. As explained in detail in Section 1 of \cite{KForms}, there is a natural $(\mathbb{A}^2_k \setminus \{(0, 0)\})$-fibration $G/U \to G/Q_s$ which can be completed to form a $2$-dimensional vector bundle equipped with a $G$-invariant symplectic form $\langle, \rangle$. With this in mind, the equivalence between $- * \overline{K(s)}$ and the symplectic Fourier-Deligne transform gives that for any $k \geq 1$ and any $f \in \mathbb{C}[\gufin]$, $x \in \gufin$,
\begin{align}\label{eqn:convfun} \left(\mathrm{tr}_{q^k}\left(f\right) * \mathrm{tr}_{q^k}\left(\overline{K(s)}\right) \right)(x) & = \frac{1}{q^k} \sum_{y \in xQ_s(\mathbb{F}_{q^k})} \psi(\langle x, y\rangle)f(y).
\end{align}

Further, it is a straightforward computation using Grothendieck's sheaf-function dictionary that for any $x \in \gufin$,
\begin{align}\label{eqn:convfun2} 
\left(\mathrm{tr}_{q^k}\left(f\right) * \mathrm{tr}_{q^k}\left(E(s)\right)\right)(x) & = \sum_{x^{-1}y \in T_s(\mathbb{F}_{q^k}) \subset X(\mathbb{F}_{q^k})} f(y).
\end{align}
We include the above formulas to note that the main computations in this paper (Propositions \ref{prop:yoko} and \ref{prop:monodrom}) can be (and indeed were, in a previous version of this paper) done explicitly on the level of functions using the formulas (\ref{eqn:convfun}) and (\ref{eqn:convfun2}) and then appealing to Proposition \ref{prop:lau}, but we will instead prove them on the level of sheaves using the six-functor formalism.

\subsection{Connection to the Yokonuma-Hecke algebra}

We will sometimes write elements of $\mathbb{C}[\gufin]$ as functions on $\gufin$, and other times as linear combinations of elements in $\gufin$. For any $g \in G(\mathbb{F}_{q^k})$, let $\delta_{g}$ be the indicator function on $\overline{g} \in \gufin$.

\begin{definition}
    Let $\mathrm{Y}_{q,k}$ (the Yokonuma-Hecke algebra) be defined as the algebra $\mathrm{End}_{G(\mathbb{F}_q)}(\mathbb{C}[\gufin])$ under composition.
\end{definition}

As explained in \cite{JuyKan}, the Bruhat decomposition $G = \coprod_{n \in N_{G(\mathbb{F}_{q^k})}(T(\mathbb{F}_{q^k}))} UnU$ gives that the standard basis of $\mathrm{Y}_{q, k}$ is parametrized by $n \in N_{G(\mathbb{F}_{q^k})}(T(\mathbb{F}_{q^k}))$. For any $s \in S$, we can define $R_s \in \mathrm{Y}_{q, k}$ (the standard basis element corresponding to $s$) as in loc. cit., and for any $t \in T(\mathbb{F}_{q^k})$ we can define $R_t$ similarly. In this way, we can identify $\mathrm{Y}_{q,k} \cong \mathbb{C}[U\backslash G/U(\mathbb{F}_{q^k})]$, where the endomorphisms $R_s, R_t \in \mathrm{Y}_{q,k}$ become identified with the indicator functions at $s$ and $t$ respectively, and composition of endomorphisms is identified with convolution.

Following \cite{JuyKan}, for any $r \in \mathbb{F}_{q^k}^\times$ we define
\begin{align*}
    h_s(r) & = \rho_s \begin{pmatrix}
r & 0\\
0 & r^{-1}
    \end{pmatrix},
\end{align*}
where $\rho_s : SL_2(\mathbb{F}_{q^k}) \to G(\mathbb{F}_{q^k})$ is the morphism corresponding to the simple root $\alpha_s$. When $t = h_s(r)$ for some $s \in S$, $r \in \mathbb{F}_{q^k}^\times$, we write $H_s(r) = R_t$. We then define for each $s \in S$ the element
\begin{align*}
    E_s & = \sum_{r \in \mathbb{F}_{q^k}^\times} H_s(r).
\end{align*}

Now we recall the original presentation of the Yokonuma-Hecke algebra, given by Yokonuma in \cite{Yoko}.
\begin{proposition}
    The Yokonuma-Hecke algebra $\mathrm{Y}_{q,k}$ is generated as an algebra by the $R_s$ ($s \in S$); and the $R_t$ ($t \in T(\mathbb{F}_{q^k})$) moreover, these generators along with the relations

\begin{align*}
    R_s^2 & = qH_s(-1) + R_sE_s\\
     \underbrace{R_sR_{s'}R_sR_{s'}\dots}_{m_{ss'}} & = \underbrace{R_{s'}R_sR_{s'}R_s \dots}_{m_{ss'}}, \quad s, s' \in S\\
     R_tR_s & = R_sR_{t'}, \quad t' = n_stn_s^{-1}, t \in T(\mathbb{F}_{q^k})\\
     R_{t_1}R_{t_2} & = R_{t_1t_2}, \quad t_1, t_2\in T(\mathbb{F}_{q^k})
\end{align*}
    define a presentation of $\mathrm{Y}_{q,k}$.
\end{proposition}

In \cite{Juyu} and \cite{JuyKan}, the authors introduced an alternative set of generators and relations; this new set of generators will be the ones directly categorified by Kazhdan-Laumon sheaves.

For any $s \in S$, let
\begin{align*}
    \Psi_s = \sum_{r \in (\mathbb{F}_{q^k})^\times} \psi(r)H_{s}(r),
\end{align*}
and let $L_s = q^{-k}(E_s + R_s\Psi_s)$.

\begin{proposition}[\cite{JuyKan}]\label{prop:yokojuyu}
    The algebra $\mathrm{Y}_{q,k}$ is generated by $L_s$, $s \in S$, and $R_t$, $t \in T(\mathbb{F}_{q^k})$, and a presentation of this algebra is given by imposing the relations
    \begin{align*}
        L_s^2 & = 1 - q^{-k}(E_s - L_sE_s)\\
        \underbrace{L_sL_tL_sL_t\dots}_{m_{st}} & = \underbrace{L_tL_sL_tL_s \dots}_{m_{st}}\\
        R_tL_s & = L_sR_{t'}, t' = sts^{-1} \quad (t \in T(\mathbb{F}_{q^k}))\\
        R_{t_1}R_{t_2} & = R_{t_1t_2} \quad (t_1, t_2 \in T(\mathbb{F}_{q^k})).
    \end{align*}
    
\end{proposition}

The action by convolution gives rise to a natural morphism $\psi_{q,k} : \widetilde{\mathrm{KL}}_{q,k} \to \mathrm{End}(\mathbb{C}[\gufin])$. Convolution of sheaves descends under the map $\mathrm{tr}_{q^k}$ to a convolution of functions in $\mathbb{C}[\gufin]$.

\begin{proposition}\label{prop:yoko}
    The image of $\psi_{q,k}$ lies in $\mathrm{Y}_{q,k}$. We have
    \begin{align*}
        \psi_{q,k}(\mathsf{a}_s) & = L_s\\
        \psi_{q,k}(\mathsf{e}_s) & = E_s.
    \end{align*}
\end{proposition}

\begin{proof}
    We first note that under the identification $\mathrm{Y}_{q,k} \cong \mathbb{C}[U \backslash G/U(\mathbb{F}_{q^k})]$, we can think of the map $\psi_{q,k}$ as having target $\mathbb{C}[U\backslash G/U(\mathbb{F}_{q^k})]$. In doing so, $\psi_{q,k}$ can be thought of as the natural inclusion $\widetilde{\mathrm{KL}}_{q,k}\hookrightarrow \mathbb{C}[U \backslash G/U(\mathbb{F}_{q^k})]$. Accordingly, it is enough to show that $\mathrm{tr}_{q^k}(\overline{K(s)}) = L_s$ and $\mathrm{tr}_{q^k}(E(s)) = E_s$ as elements of $\mathbb{C}[U \backslash G /U(\mathbb{F}_{q^k})]$.

    By the definition of $E(s) = i_{T_s,!}\mathbb{C}$, this is clear for $E(s)$. It remains to show that $\mathrm{tr}_{q^k}(\overline{K(s)}) = L_s$ for any $s \in S$. Since in this computation we only consider a single simple reflection, it is enough to check this in the case where $G = \mathrm{SL}_2$. In this case, let $B$ be the Borel subgroup of upper-triangular matrices with unipotent radical $U$, and let $p_s : U \backslash \mathrm{SL}_2/U \to \mathbb{G}_a$ be the map
    \begin{align}
        \begin{pmatrix}
            a & b\\
            c & d
        \end{pmatrix} \mapsto -c.
    \end{align}
    It follows from the description of $\overline{K(s)}$ in Section 5.2 of \cite{P} that $\overline{K(s)} = p_s^*\mathcal{L}_\psi[2](1)$. Now note that any element of $N_{G(\mathbb{F}_{q^k})}(T(\mathbb{F}_{q^k}))$ can be written as either $t$ or $st$ for some $t \in T(\mathbb{F}_{q^k})$. For any such element, we have that
    \begin{align}
        L_s(t) & = q^{-k}E_s(t) + q^{-k}R_s\Psi_s(t)\\
        & = q^{-k} + 0\\
        L_s(st) & = q^{-k}E_s(st) + q^{-k}R_s\Psi_s(st)\\
        & = 0 + q^{-k}\psi(t).
    \end{align}
    Similarly, by our description of $\overline{K(s)}$ and  the six-functor formalism for functions, we have that, as elements of $\mathbb{C}[U\backslash G/U(\mathbb{F}_{q^k})]$, $\mathrm{tr}_{q^k}(\overline{K(s)}) = q^{-k}\mathrm{tr}_{q^k}(\mathcal{L}_{\psi})\circ p_s$. So, when identifying the target of $\psi_{q,k}$ with $\mathbb{C}[U\backslash G/U(\mathbb{F}_{q^k})]$, for any $t \in T(\mathbb{F}_{q^k})$ we have
    \begin{align}
        \psi_{q,k}(\mathsf{a}_s)(t) & = q^{-k}\psi(0)= q^{-k},\\
        \psi_{q,k}(\mathsf{a}_s)(st) & = q^{-k}\psi(t),
    \end{align}
    which guarantees that $\psi_{q,k}(\mathsf{a}_s) = L_s$. 
\end{proof}

\subsection{Algebras with a generic parameter}\label{subsec:genparam}

We now use Proposition \ref{prop:yoko} to prove our main result.
\begin{proof}[Proof of Theorem \ref{thm:mainthm}]
    Consider the diagram
    $$\begin{tikzcd}
     \widetilde{\mathrm{KL}}(v) \arrow[dr, "\mathrm{tr}"'] & & \mathcal{E}(v) \arrow[dl, "i"]\\
      & \prod_{k \geq 1} \mathbb{C}[\gufin] &,
    \end{tikzcd}$$
    where $i$ is assembled from the composition of the natural maps
    \begin{align*}
        \mathcal{E}(v) \to \mathcal{E}_{q^k} \to \mathrm{Y}_{q,k} \hookrightarrow \mathbb{C}[\gufin],
    \end{align*}
    where $\mathcal{E}(v) \to \mathcal{E}_{q^k}$ is the specialization map at $v = q^{k/2}$. 
    
    The fact that $\mathrm{tr}$ is injective follows from Proposition \ref{prop:lau}. We now argue that $i$ is injective. In Section 3.4 of \cite{MarYoko}, which relies on the proof from the end of Section 2.1 of loc. cit., it is stated that the natural map from $\mathcal{E}_{q'}$ (called $\mathcal{C}_{W,q}^R$ in loc. cit.) to $\mathrm{Y}_{q'}$ is an embedding for ``generic" $q'$. Following the proof of this fact in Section 2.1 of loc. cit., one recovers a more precise statement: there exists some $r_0$ depending on $W$ such that the natural map $\mathcal{E}_{q'} \to \mathrm{Y}_{q'}$ sending $g_s \to -L_s$ is injective whenever $q' - 1$ has a prime factor larger than $r_0$. For $q = p^k$ any prime power, there exists an $a$ for which $q' = q^a$ satisfies this condition. Indeed, let $p'$ be a prime number greater than $r_0$ and distinct from $p$, and set $a = p' - 1$; then by Fermat's little theorem, $p'$ divides $q^a - 1$. Further, we can take large enough $a$ so that the specialization map $v \mapsto q^{a/2}$ taking $\mathcal{E}(v) \to \mathcal{E}_{q^a}$ is itself injective.

    Now note that Proposition \ref{prop:yoko} (applied to all powers of $q$ at once) shows that the images of $\mathrm{tr}$ and $i$ are equal, and that 
    \begin{align*}
        \mathrm{tr}(\mathsf{a}_s) = i(-g_s),\\
        \mathrm{tr}(\mathsf{e}_s) = i(e_s).
    \end{align*}

    Since we have just identified both $\mathrm{KL}(v)$ and $\mathcal{E}(v)$ as the same subalgebra of $\prod_{k\geq 1}\mathbb{C}[\gufin]$, it follows that the map
    \begin{align*}
        \psi : \widetilde{\mathrm{KL}}(v) \to \mathcal{E}(v)
    \end{align*}
    defined on generators by $\psi(\mathsf{a}_s) = -g_s$, $\psi(\mathsf{e}_s) = e_s$ is a well-defined algebra isomorphism. The same proof works to show that there is a well-defined isomorphism $\psi' : \mathrm{KL}(v) \to \mathcal{C}(v)$ with $\psi'(\mathsf{a}_s) = -g_s$.
\end{proof}

\subsection{Connection to Polishchuk's cubic relation}

As we have explained, since Theorem \ref{thm:mainthm} gives a full description of the algebra $\mathrm{KL}(v)$ can be thought of as an enhancement of previous results on some relations satisfied by the generators of this algebra.

Namely, in addition to the braid relations shown by Kazhdan and Laumon in \cite{KL}, Polishchuk found the additional relation
\begin{equation}\label{eqn:pol}
    (\mathsf{a}_s^2 - 1)(\mathsf{a}_s + v^2) = 0
\end{equation}
for any $s \in S$. We now explain how this relation can be recovered from our description. By Theorem \ref{thm:mainthm}, recalling that $\mathsf{a}_s$ can be identified with $-g_s$, it is enough to show that $(g_s^2 - 1)(g_s - v^2) = 0$ in $\mathcal{E}(v)$, so we can work with the relations in Definition \ref{def:evcv}. The last relation shows that $g_s^2 - 1 = (v^2 - 1)e_s(1 + g_s)$, so (\ref{eqn:pol}) becomes equivalent to $(g_s - v^2)(1 + g_s)e_s = 0$. This is is then obtained by multiplying the last relation in Definition \ref{def:evcv} by $e_s$.

An informal interpretation of the relation is that (\ref{eqn:pol}) is the least common multiple of the two expressions $\mathsf{a}_s^2 - 1$ and $(\mathsf{a}_s - 1)(\mathsf{a}_2 + v^2)$, the former of which is present in the Weyl group, and the latter in the Hecke algebra. Since we expect any relation in $\mathrm{KL}(v)$ to act trivially on $\mathbb{C}[\gufin]$, the relation (\ref{eqn:pol}) is then natural since it is the minimal relation among powers of $\mathsf{a}_s$ which acts trivially on $\mathbb{C}[\gufin]_\theta$ both when $s \not\in W_{\theta}^\circ$ and when $s \in W_{\theta}^\circ$ (with the Weyl group and Hecke algebra relations holding in these two situations respectively). In this sense, the algebra $\mathrm{KL}(v)$ can be thought of as the ``minimal" algebra containing all $\mathsf{a}_s$, $s \in S$ which encodes the relations of all monodromic Hecke algebras at once.

\section{Connection to monodromic Hecke categories}

\subsection{The monodromic Hecke algebra}

For a $W$-orbit $\mathfrak{o}$ of character sheaves on $T$, Lusztig and Yun define in Section 3.14 of \cite{LY} a unital associative $\mathbb{Z}[v, v^{-1}]$-algebra with generators $A_w$ (called $T_w$ in loc. cit.) for $w \in W$ and $1_{\mathcal{L}}$ ($\mathcal{L} \in \mathfrak{o}$) satisfying the relations

\begin{eqnarray*}
\notag&&1_{\mathcal{L}}1_{\mathcal{L}'}=\delta_{\mathcal{L},\mathcal{L}'}1_{\mathcal{L}},\quad\textup{ for }\mathcal{L},\mathcal{L}'\in\mathfrak{o};\\
\label{lengths add}&&A_wA_{w'}=A_{ww'},\quad\text{ if }w,w'\in W\text{ and } \ell(ww')=\ell(w)+\ell(w');\\
\notag&&A_w1_{\mathcal{L}}=1_{w\mathcal{L}}A_w, \quad\text{ for }w\in W,\mathcal{L}\in\mathfrak{o};\\
\label{quad rel}&&A_s^2=v^2A_1+(v^2-1)\sum_{\mathcal{L};s\in W_{\mathcal{L}}^\circ}A_s1_{\mathcal{L}},\text{  for simple reflections }s\in W;\\
\notag&&A_1=1=\sum_{\mathcal{L}\in\mathfrak{o}}1_{\mathcal{L}}.
\end{eqnarray*}

As in \cite{LY}, we set $\tilde{A}_w = v^{-\ell(w)}A_w \in \mathbf{H}_{\mathfrak{o}}$. In section 3.14 of \cite{LY}, the authors discuss a connection between this algebra and $K_0(\oplus_{\mathcal{L}, \mathcal{L}' \in \mathfrak{o}}(_{\mathcal{L'}}\mathcal{D}_{\mathcal{L}}))$, where the elements $\tilde{A}_{w}1_{\mathcal{L}}$ become associated with classes in the Grothendieck group of the standard monodromic sheaves $\Delta(w)_{\mathcal{L}}$ defined and studied in loc. cit. As a result, we can identify $\mathbf{H}_{\mathfrak{o}}$ with a subalgebra of $K_0(U \backslash G/U)$ under convolution generated by the $[\Delta(w)_{\mathcal{L}}]$ for $w \in W$, $\mathcal{L} \in \mathfrak{o}$. 

\subsection{Monodromic sheaves from Kazhdan-Laumon sheaves}

\begin{proposition}\label{prop:monodrom}
    Convolution with $\Delta(e)_{\mathcal{L}}$ induces a surjection  $\pi_{\mathcal{L}}$ from $\mathrm{KL}(v) \to \mathbf{H}_{\mathfrak{o}}$. We have
    \begin{align}
        \pi_{\mathcal{L}}(\mathsf{a}_s) = \begin{cases}
        -v\tilde{A}_{s}^{-1}1_{\mathcal{L}} & s \in W_{\mathcal{L}}^\circ\\
        -\tilde{A}_s1_{\mathcal{L}} & s \not\in W_{\mathcal{L}}^\circ
        \end{cases}
    \end{align}
\end{proposition}

\begin{proof}
    Since the claim only concerns a single simple reflection, it is enough to prove the result for $G = \mathrm{SL}_2$. Let $j_s : U \backslash BsB/U \to U \backslash G /U$ and $j_e : U\backslash B/U \to U \backslash G/U$ be the inclusions of the open and closed Bruhat strata. Recall the projection to $T$ given by $\mathrm{pr}_s : U \backslash BsB/U \to T$ as defined in Section \ref{sec:klshconv}.

    Convolving the usual open-closed distinguished triangle for the maps $j_s$ and $j_e$ with $\Delta_{\mathcal{L}}(e)$ gives the distinguished triangle
    \begin{align}
        (j_{s!}j_{s}^*\overline{K(s)}) * \Delta_{\mathcal{L}}(e) \to \overline{K(s)} * \Delta_{\mathcal{L}}(e) \to (j_{e*}j_{e}^*\overline{K(s)}) * \Delta_{\mathcal{L}}(e).
    \end{align}
Since $U\backslash B/U \cong T$ and so $\Delta_{\mathcal{L}}(e) = j_{e!}(\mathcal{L}) = j_{e*}(\mathcal{L})$, proper base change and the projection formula imply
\begin{align}
    (j_{e*}j_{e}^*\overline{K(s)}) * \Delta_{\mathcal{L}}(e) & = (j_{e*}\mathcal{L})\otimes R\Gamma_c(T, \mathcal{L}),\label{eqn:closed}\\
    (j_{s!}j_{s}^*\overline{K(s)}) * \Delta_{\mathcal{L}}(e) & = j_{s!}\mathrm{pr}_s^*(\mathcal{L}[2](1))\otimes R\Gamma_c(T, \mathcal{L} \otimes \mathcal{L}_{\psi}).\label{eqn:open}
\end{align}
In this rank-one situation, if $s \in W_{\mathcal{L}}^\circ$, then $\mathcal{L}$ is constant. In this case, in (\ref{eqn:closed}) we get the tensor product of $\Delta_{\mathcal{L}}(e)$ itself with the compactly supported cohomology of $T \cong \mathbb{G}_\mathrm{m}$, whose class in $K_0$ is $(v^2 - 1)$. The factor $R\Gamma_c(T, \mathcal{L} \otimes \mathcal{L}_{\psi})$ in  (\ref{eqn:open}) is $1$-dimensional and concentrated in cohomological degree $1$, and has weight zero by Remark 4.4 of \cite[Applications de la formule des traces aux sommes trigonom{\'e}triques]{DCoh}. Taking the class in the Grothendieck group, and then summing the terms obtained from (\ref{eqn:closed}) and (\ref{eqn:open}), we then get $-v\tilde{A_s}^{-1}1_{\mathcal{L}}$.

Similarly, if $s \not\in W_{\mathcal{L}}^\circ$, then in this case $\mathcal{L}$ is not constant. So the term (\ref{eqn:closed}) is zero since $R\Gamma_c(T, \mathcal{L}) = 0$. Further, the same remark from \cite{DCoh} gives that $R\Gamma_c(T, \mathcal{L} \otimes \mathcal{L}_{\psi})$  is still concentrated in cohomological degree $1$, is one-dimensional, but this time has weight $1$. This leaves only a single term, whose class in the Grothendieck group gives $-\tilde{A_s}1_{\mathcal{L}}$ in the monodromic Hecke algebra.
\end{proof}

From now on, let $\pi = \pi_{\mathcal{L}_{\mathrm{triv}}}$ be the surjection $\mathrm{KL}(v) \to \mathcal{H}$.

\subsection{Lifting the Kazhdan-Lusztig basis of the Hecke algebra}

We now recall the Kazhdan-Lusztig basis $C_w$ of the Hecke algebra and its interpretation in terms of simple perverse sheaves $\mathrm{IC}(w)$ on $G/U$. The Kazhdan-Lusztig basis $\{C_w\}_{w \in W}$ of the Hecke algebra $\mathcal{H}$ is a basis which is uniquely characterized by the conditions that $C_w \in \tilde{A}_w + \sum_{y \leq w} v\mathbb{Z}[v]\tilde{A}_y$ and $\overline{C}_w = C_w$.

The following lemma is well-known (for instance, c.f. (3.5.3) of \cite{W}):
\begin{lemma}\label{lem:cmult}
    For any $w, w' \in W$ with $\ell(ww') = \ell(w) + \ell(w')$, we have 
    \begin{align*}
        C_{w}C_{w'} - C_{ww'} \in \mathrm{span}_{\mathbb{Z}} \{C_y\}_{y < ww'}
    \end{align*}
\end{lemma}

There is a well-known geometric interpretation of the generators $\tilde{A}_w$ and $C_w$ of $\mathcal{H}$. Indeed, there are standard sheaves $\Delta(w)$ and $\mathrm{IC}(w)$ on $G/U$ which categorify these generators respectively, c.f. Chapter 7 of \cite{Achar} for a very clear exposition of this setup in the case of $G/B$; a version of this with a different normalization relevant to the Kazhdan-Laumon context was explained in Section 4.1 of \cite{KLO} for $G/U$. With this in mind, Lemma \ref{lem:cmult} shows that the span of $\{[\mathrm{IC}(w)]\}_{w \in W}$ is a subalgebra of $K_0(U \backslash G/U)$.

\begin{proposition}[Proposition 3.8 in \cite{MarYoko}]
    The assignment
    \begin{align*}
        \overline{g_s} & = g_s^{-1}\\
        \overline{e_s} & = e_s\\
        \overline{v} & = v^{-1}
    \end{align*}
    extends to well-defined ring automorphism of $\mathcal{E}(v)$. We will use the same bar notation to denote the corresponding involution on $\mathcal{C}(v)$, $\mathrm{KL}(v)$, and $\widetilde{\mathrm{KL}}(v)$ using the isomorphisms of the previous section.
\end{proposition}

\begin{proposition}
    There exists a linearly independent set $\{\mathsf{c}_{w}\}_{w \in W}$ of $\mathrm{KL}(v)$ satisfying $\overline{\mathsf{c}_w} = \mathsf{c}_w$. It has the property that
    \begin{align*}
        \pi(\mathsf{c}_{w}) = C_{w}.
    \end{align*}
\end{proposition}

\begin{proof}
    We will first show that for any $w \neq 1$, $[\mathrm{IC}(w)] \in K_0(G/U)$ actually lies in $\mathrm{KL}(v)$; once we do that, we will define $\mathsf{c}_1 = 1$ and $\mathsf{c}_w = [\mathrm{IC}(w)]$ for any $w \neq 1$, and then check the properties in the proposition.

    We begin in the case where $w = s \in S$. In this case, we define
    \begin{align}
        \mathsf{c}_s & = \frac{\mathsf{a}_s^{2} - 1}{v - v^{3}}.
    \end{align}
    It is easy to check by working on the level of functions on $\gufin$ that $\mathsf{c}_s$ as defined above is equal to $[\mathrm{IC}(s)]$.
    
    We have
    \begin{align*}
        \pi(\mathsf{c}_s) & = \frac{v^{2}\tilde{A}_{s}^{-2} - 1}{v - v^{3}}\\
        & = \frac{v^{2}((v^{-1}-v)\tilde{A}_s + (v^{-2}-1+v^2)) - 1}{v - v^{3}}\\
        & = \tilde{A}_s - v = c_s.
    \end{align*}

    Using the relation $(\mathsf{a}_s^2 - 1)(\mathsf{a}_s + v^{2}) = 0$, it is a straightforward computation that
    \begin{align*}
        \mathsf{a}_s^{-2} - 1 = v^{-4}(1 - \mathsf{a}_s^2),
    \end{align*}
    which implies
    \begin{align*}
        \overline{\mathsf{c_s}} & = \frac{\mathsf{a}_s^{-2} - 1}{v^{-1} - v^{-3}}\\
        & = \frac{1 - \mathsf{a}_s^2}{v^4(v^{-1} - v^{-3})}\\
        & = \mathsf{c}_s.
    \end{align*}
    
    For the case of general $w \in W$, suppose for induction that $[\mathrm{IC}(y)] \in \mathrm{KL}(v)$ for all $y$ with $\ell(y) < \ell(w)$. Choose some $s \in S$ such that $\ell(sw) < \ell(w)$. By Lemma \ref{lem:cmult} and by our induction hypothesis along with the fact that $\mathrm{KL}(v)$ is a subalgebra, this implies $[\mathrm{IC}(sw)] \in \mathrm{KL}(v).$ Now we can define $\mathsf{c}_1 = 1$ and $\mathsf{c}_w = [\mathrm{IC}(w)]$ for all $w \neq 1$, and this clearly defines a linearly independent subset of $\mathrm{KL}(v)$.
    
    We now show that $\overline{\mathsf{c}_w} = \mathsf{c}_w$. Indeed, this is trivial for $\mathsf{c}_1$, and checked this for $\mathsf{c}_s$ above for any $s \in S$. By the definition of the duality on $\mathrm{KL}(v)$, it is multiplicative and $\mathbb{Z}$-linear, and so $\overline{\mathsf{c}_w} = \mathsf{c}_w$ follows again by induction on $\ell(w)$ using Lemma \ref{lem:cmult}

    Finally, we show that $\pi(\mathsf{c}_w) = C_w$ for all $w \in W$. First note that clearly the bar involution has the property that it lifts the usual bar involution on the Hecke algebra, so $\overline{\pi(\mathsf{c}_w)} = \pi(\mathsf{c}_w)$. Now once again, Lemma \ref{lem:cmult} implies the other defining property of $C_w$; this means we must have $\pi(\mathsf{c}_w) = C_w$. Alternatively, we could have simply observed that by the definition of $\mathsf{c}_w$, $\pi(\mathsf{c}_w)$ is $[\mathrm{IC}(w)] \in \mathcal{H} \cong K_0(B\backslash G /U)$ which corresponds to the canonical basis element $C_w$.
\end{proof}

\section{Dimension formula}

In this section, assume that we are in Type $A_n$, i.e. $G = \mathrm{SL}_{n+1}$. This is the case where $\mathcal{E}(v)$ becomes the usual algebra of braids and ties as originally defined in \cite{AJ}. The subalgebra $\mathcal{C}_{A_n}(v) \subset \mathcal{E}_{A_n}(v)$ was studied in more detail in this case in Section 4 of \cite{MarYoko}. In loc. cit., the author used a computer to compute the sequence of dimensions of $\mathcal{C}_{A_n}(v)$ for a few rational values of $v$, obtaining the table below.

\begin{figure}[ht]
\centering
    \begin{tabular}{|c|c|c|c|c|}
\hline
     $n$ &  1 & 2 & 3 & 4\\
\hline
     $\dim \mathcal{C}_{A_n}(v)$ &  3 & 20 & 217 & 3364\\
\hline
\end{tabular}
\end{figure}

In this section, we give an explicit formula for $\dim \mathcal{C}_{A_n}(v)$ for all values of $n$, extending this table.

\subsection{A formula in terms of reflection subgroups of $W$}

\begin{definition}
    We call a reflection subgroup $R' \subset W$ contiguous if its Dynkin type is connected. For any reflection subgroup $R \subset W$, we can write $R = \prod_{i=1}^\ell R_i'$ for some contiguous reflection subgroups $R_i'$ uniquely up to permutation of the $R_i'$; we call these the contiguous components of $R$. 
\end{definition}

\begin{definition}
    For any reflection subgroup $R$ of $W$, let $J_R$ be the vector subspace of $\mathcal{E}(v)$ spanned by all terms of the form
    \begin{align}\label{eqn:jrdef}
        g_w(1 + g_{r_1}) \dots (1 + g_{r_\ell}),
    \end{align}
    where $r_1, \dots, r_\ell$ are reflections lying in distinct contiguous components of $R$.
\end{definition}

\begin{lemma}
    For any reflection subgroup $R$ of $W$, $J_{R}e_{R} \subset \mathcal{C}(v)$. 
\end{lemma}

\begin{proof}
    By induction on the number of contiguous components of $R$, it is enough to show that for any contiguous reflection subgroup $R$ (i.e. one of type $A_k$ for some $k \leq n$) of $W$ and any reflection $r \in R$, we have $(g_r + 1)e_R \in \mathcal{C}(v)$.

    Suppose for induction (the case $k = 1$ is tautological) that this result holds for all reflection subgroups of type $A_{k'}$, $k < k$. We will first show that for any reflections $r, r' \in R$ for which $\langle r, r'\rangle$ is itself a reflection subgroup of type $A_2$ that $(g_r' + 1)(g_r + 1)e_{\langle r, r'\rangle} \in \mathcal{C}(v)$. To do this, using the fact that all reflection subgroups $R$ of type $A_2$ are conjugate, we can choose a set of reflections $\{r_1, \dots, r_{k-2}\}$ in $R$ such that $\{r_1, \dots, r_{k-2}, r', r\}$ generates $R$ and for which $\{r_1, \dots, r_{k-2}, r'\}$ is of type $A_{k-1}$. So applying our induction hypothesis, $(g_{r'} + 1)e_{\langle r_1, \dots, r_{k-2}, r'\rangle} \in \mathcal{C}(v)$, and so $(g_{r'} + 1)(g_r + 1)e_{R}$ is too. Similarly we can reverse the roles of $r$ and $r'$ to obtain $(g_r + 1)(g_{r'} + 1)e_R$ as well.

    Now using again the conjugacy in $R$ of reflection subgroups of type $A_2$, let $y \in R$ be such that $yry^{-1} = s_1, yr'y^{-1} = s_2$. Then we can deduce that
    \begin{align*}
        g_y^{-1}(g_{s_1} + 1)(g_{s_2} + 1)e_{\langle s_1, s_2\rangle}g_y, g_{y}^{-1}(g_{s_2} + 1)(g_{s_1} + 1)e_{\langle s_1, s_2\rangle}g_y \in \mathcal{C}(v).
    \end{align*}
    Since the $g = g_{s_1}e_{\langle s_1, s_2\rangle}$ or $g = g_{s_2}e_{\langle s_1, s_2\rangle}$ satisfy the usual quadratic Hecke relation $(g - v^2)(g + 1) = 0$, it is a straightforward computation in the Hecke algebra of type $A_2$ to observe that
    \begin{align*}
        & (g_{s_1} - 1)e_{\langle s_1, s_2\rangle}, (g_{s_2} - 1)e_{\langle s_1, s_2\rangle}\\
        &  \in \mathrm{span} \{g_2(g_{s_1} + 1)(g_{s_2} + 1)e_{\langle s_1, s_2\rangle}, g_w(g_{s_2} + 1)(g_{s_1} + 1)e_{\langle s_1, s_2\rangle}\}_{w \in \langle s_1, s_2\rangle}.
    \end{align*}
    By conjugating back by $g_y$, we deduce that $(g_r + 1)e_{\langle r, r'\rangle}, (g_{r'} + 1)e_{\langle r, r'\rangle} \in \mathcal{C}(v)$.

    Now in the general case, note that for any reflection $r \in W$, $(g_r + 1)e_{\langle r\rangle} \in \mathcal{C}(v)$ by the quadratic relation satisfied by $g_r$. So for the rank $k$ reflection subgroup $R$ in question, we have
    \begin{align*}
        (g_{r_{i_1}} + 1)e_{r_{i_1}}\dots (g_{r_{i_k}} + 1)e_{r_{i_k}} = (g_{r_{i_1}} + 1)\dots (g_{r_{i_k}} + 1)e_R \in \mathcal{C}(v)
    \end{align*}
    for $\{r_1, \dots, r_k\}$ a generating set of reflections for $R$ and $(i_{j})_{j=1}^k$ any permutation of the indices. By then applying our rank $2$ result to pairs of reflections in this generating set, the result that $(g_r + 1)e_R \in \mathcal{C}(v)$ for any $r \in R$ follows too by induction.
\end{proof}

\begin{lemma} As vector spaces over $\mathbb{C}(v)$, we have
    \begin{align}\label{eqn:brjr}
    \mathcal{C}(v) &= \bigoplus_{R\subset W} J_{R}e_{R}.
\end{align}
where $R$ ranges over all reflection subgroups of $W$. 
\end{lemma}

\begin{proof}
    The algebra $\mathcal{C}(v)$ is, by definition, spanned by words in $\{g_s^{\pm 1}\}_{s \in S}$. Using the quadratic relation $g_s^2 = 1 + (v^2 - 1)(1 + g_s)e_s$, we then see that any element of $\mathcal{C}(v)$ is in the $\mathbb{C}(v)$-span of words in the set of terms $\{g_s^{-1}\}_{s \in S}\cup \{(1 + g_s)e_s\}_{s \in S}$. For any $s, t \in S$, we have
    \begin{align*}
        (1 + g_t)e_tg_s^{-1} & = g_s^{-1}(1 + g_sg_tg_s^{-1})e_{tst}.
    \end{align*}
    By repeatedly using this relation along with the quadratic relation on a given word to move all $g_s$ terms to the left and reducing any powers $g_s^2$, $s \in S$ which occur we are left with a spanning set consisting of words of the form
    \begin{align}\label{eqn:messyform}
        g_w^{-1}(1 + b_1g_{t_1}b_1^{-1})e_{b_1t_1b_1^{-1}} \dots (1 + b_\ell g_{t_\ell}b_{\ell}^{-1}) e_{b_\ell t_\ell b_\ell^{-1}},
    \end{align}
    where $b_i$ is some product of the $\{g_s\}_{s \in S}$ and each $t_i$ lies in $S$.
    We want to reduce this further to show that $\mathcal{C}(v)$ is spanned by the set of elements of the form
    \begin{align}\label{eqn:cleanform}
        g_w^{-1}(1 + g_{r_1})\dots (1 + g_{r_\ell})e_{R}
    \end{align}
    for $R$ a reflection subgroup and $r_i$ reflections, each of which lies in a distinct contiguous component of $R$.

    To reduce terms of the form (\ref{eqn:messyform}) to those of the form (\ref{eqn:cleanform}), by induction on the number of terms it is enough to show that each term of the form $(1 + bg_{t}b^{-1})e_{b^{-1}tb}$ can be written in the form (\ref{eqn:cleanform}). This fact follows from induction on the number of terms in $b$ once one shows it for $b = g_s$ for some $s \in S$. In this case,
    \begin{align*}
        (1 + g_sg_tg_s^{-1})e_{sts} & = (1 + g_{sts})e_{sts} + (v^{-2} - 1)g_sg_t(1 + g_s)e_{\langle s, t\rangle},
    \end{align*}
    which is of the desired form. By following the induction as mentioned and using this relation, one obtains that all elements of $\mathcal{C}(v)$ can be written as a $\mathbb{C}(v)$-linear combination of elements of the form (\ref{eqn:cleanform}).

    To see that this sum is direct, we note that each $e_R$, when specialized to $\mathbb{C}[\gufin]$, can be thought of as an idempotent projection to the subspace of functions monodromic with respect to all $\theta$ for which $R \subset W_{\theta}^\circ$. So every $R$ determines a subset $M_R$ of multiplicative characters of $T(\mathbb{F}_{q^k})$ such that $e_R \cdot \mathbb{C}[\gufin]_{\theta} = \mathbb{C}[\gufin]_{\theta}$, with $e_R \cdot \mathbb{C}[\gufin]_{\theta'} = 0$ for every $\theta' \not\in M_R$. There is a direct sum decomposition $\mathbb{C}[\gufin] = \oplus_\theta \mathbb{C}[\gufin]_\theta$, and recall that we can choose $k$ large enough so the specialization map $\mathcal{C}(v) \to \mathbb{C}[\gufin]$ is injective by Section \ref{subsec:genparam}. So since each $R$ determines a distinct such subset $M_R$, there can be no nontrivial linear relations between elements of $\mathcal{C}(v)e_R$ for distinct $R$.
\end{proof}

\begin{definition}
    For any subset $I \subset S$, let $D_I$ be the number of elements $w \in W$ such that for each contiguous block $I' \subset I$, there is a simple reflection $s \in I'$ for which $\ell(ws) < \ell(w).$
\end{definition}

Since we are in Type $A$, any reflection subgroup is conjugate to a parabolic subgroup $P_I \subset W$ corresponding to some $I \subset S$. 
\begin{lemma}
    Let $R$ be any reflection subgroup of $W$ conjugate to $P_I$. Then $\dim J_Re_{R} = D_I$.
\end{lemma}

\begin{proof}
    By invariance of dimension under conjugation, it is enough to prove this for $R = P_I$ for some $I \subset S$. In this case, we recall the fact that for $s \in P_I$, each $g_sP_I$ satisfies a quadratic relation, ie.. $(g_s - v^2)(g_s + 1)e_{P_I} = 0$. This implies that $\dim J_Re_R$ is equal to the dimension of the left ideal of the group algebra $\mathbb{C}[W]$ generated by
    \begin{align}
        (1 + r_1) \dots (1 + r_k) \in \mathbb{C}[W],
    \end{align}
    where each $r_i$ is in a distinct contiguous component of $I$. If we let $J_I'$ be this ideal, it is straightforward to check that $\mathbb{C}[W]/J_I'$ is spanned by equivalence classes $\overline{w}$ of $w \in W$ such that there exists a contiguous block $I'$ of $I$ for which no reduced expression for $w$ has an element of $I'$ as a suffix. Thus, the dimension of $\mathbb{C}[W]/J_I'$ is exactly the complement of the set described in the statement of the lemma, and so the size of this latter set agrees with $\dim J_I' = \dim J_Ie_I = \dim J_Re_R$.
\end{proof}

\subsection{An explicit formula}

\begin{definition}
    Let $P(n)$ be the set of all partitions $(\lambda_i)_{i=1}^k$ for which
    \begin{align}
        \left(\sum_{i=1}^k \lambda_i \right) + k - 1 \leq n.
    \end{align}
    We can identify $P(n)$ with a choice of contiguous blocks $I \subset S$ of the Dynkin diagram of type $A_n$ whose size is decreasing, e.g. when $n = 7$, the partition $(2, 2, 1)$ corresponds to the choice pictured here.
    \begin{center}
\resizebox{!}{0.23cm}{
\begin{tikzpicture}[line width=0.045cm,cc1/.style={minimum size=0.75cm,path picture={
\fill (0,0) circle[radius=2mm];},node contents={}},
cc2/.style={circle,draw,inner sep=0pt,minimum size=0.75cm,path picture={
\fill (0,0) circle[radius=2mm];},node contents={}},scale=1.5]
\path (0,0) node[cc2];
\path (1,0) node[cc2];
\path (2,0) node[cc1];
\path (3,0) node[cc2];
\path (4,0) node[cc2];
\path (5,0) node[cc1];
\path (6,0) node[cc2];
\draw (0,0) -- (1,0);
\draw (1,0) -- (2,0);
\draw (2,0) -- (3,0);
\draw (3,0) -- (4,0);
\draw (4,0) -- (5,0);
\draw (5,0) -- (6,0);
\end{tikzpicture}}
\end{center}
Given any subset $I \subset S$, we can reorder its contiguous blocks so that their size is decreasing, and define $\lambda^I$ to be the associated partition. Finally, for any $i \geq 1$ let $n_i(I)$ be the number of parts of $\lambda^I$ which are equal to $i$.
\end{definition}

\begin{lemma}[\cite{How}]\label{lem:how}
    Given any subset $I \subset S$, the normalizer of the parabolic subgroup $W_I$ of $W$ generated by the simple reflections in $I$ has order
    \begin{align}
        \left(n + 1 - \sum_i (i + 1)n_i(I)\right)! \cdot \prod_{i} n_i(I)!(i + 1)!^{n_i(I)}.
    \end{align}
\end{lemma}

\begin{lemma}\label{lem:di}
    For any $I \subset S$,
    \begin{align}
        D_I = \sum_{T \subset \{\lambda_i^I\}_{i}} (-1)^{|T|} \frac{(n+1)!}{\prod_{\lambda_i^I \in T} (\lambda_i^I+1)!}
    \end{align}
    where the sum is taken over all subpartitions of $\lambda^I$, including $\lambda^I$ itself.
\end{lemma}
\begin{proof}
    Indeed, this follows by the principle of inclusion-exclusion, since the quantity $$\frac{(n + 1)!}{\prod_{\lambda_i^I \in T}(\lambda_i + 1)!}$$ is equal to the number of elements $w \in W$ having no reduced expression ending with any of the simple reflections in any of the blocks $\lambda_i^I$ in a given subset $T$ of blocks.
\end{proof}

\begin{theorem}
    For any $I \subset S$, let $R_I$ be the number of reflection subgroups of $W$ conjugate to $P_I$ (which is $|W|$ divided by the quantity in Lemma \ref{lem:how}), and let $D_I$ be as above. Then the dimension of $\mathcal{C}(v)$ in Type $A_n$ is given by
    \begin{align}
        \dim \mathcal{C}(v) & = \sum_{I \subset S} R_I \cdot D_I.
    \end{align}
\end{theorem}

\subsection{Examples and comparison with \cite{MarYoko}}

We conclude by giving three examples of how these formulas work in practice by verifying the computations done by Marin in Section 4.3 of \cite{MarYoko} in types $A_2, A_3, A_4$. Note that the sequence of dimensions obtained in loc. cit. in these cases is $20, 217, 3364$ (it is clear that $\dim \mathcal{C}_{A_1}(v) = 3$). We will now show how this sequence arises from our computations.

In Figure \ref{fig:smallrank}, we write $N_I$ for the dimension of the normalizer of $W_I$, i.e. the quantity given in Lemma \ref{lem:how}. Then $R_I$ is obtained by $R_I = (n + 1)!/N_I$ as explained previously, with $D_I$ being given by the formula in Lemma \ref{lem:di}.

\begin{figure}[ht]
    \centering
    \begin{tabular}{|c|c|c|c|}
\hline
$I$ & $N_I$ & $R_I$ & $D_I$
\\
\hline
\resizebox{!}{0.23cm}{
\begin{tikzpicture}[line width=0.045cm,cc1/.style={minimum size=0.75cm,path picture={
\fill (0,0) circle[radius=2mm];},node contents={}},
cc2/.style={circle,draw,inner sep=0pt,minimum size=0.75cm,path picture={
\fill (0,0) circle[radius=2mm];},node contents={}},scale=1.5]
\path (0,0) node[cc2];
\path (1,0) node[cc2];
\draw (0,0) -- (1,0);
\end{tikzpicture}} & $6$ & $1$ & $5$\\
    \hline
    \resizebox{!}{0.23cm}{
\begin{tikzpicture}[line width=0.045cm,cc1/.style={minimum size=0.75cm,path picture={
\fill (0,0) circle[radius=2mm];},node contents={}},
cc2/.style={circle,draw,inner sep=0pt,minimum size=0.75cm,path picture={
\fill (0,0) circle[radius=2mm];},node contents={}},scale=1.5]
\path (0,0) node[cc2];
\path (1,0) node[cc1];
\draw (0,0) -- (1,0);
\end{tikzpicture}} & $3$ & $2$ & $3$ \\
    \hline
    \resizebox{!}{0.23cm}{
\begin{tikzpicture}[line width=0.045cm,cc1/.style={minimum size=0.75cm,path picture={
\fill (0,0) circle[radius=2mm];},node contents={}},
cc2/.style={circle,draw,inner sep=0pt,minimum size=0.75cm,path picture={
\fill (0,0) circle[radius=2mm];},node contents={}},scale=1.5]
\path (0,0) node[cc1];
\path (1,0) node[cc1];
\draw (0,0) -- (1,0);
\end{tikzpicture}} & $6$ & $1$ & $6$ \\
\hline
\end{tabular}
\quad 
\begin{tabular}{|c|c|c|c|}
\hline
$I$ & $N_I$ & $R_I$ & $D_I$
\\
\hline
\resizebox{!}{0.23cm}{
\begin{tikzpicture}[line width=0.045cm,cc1/.style={minimum size=0.75cm,path picture={
\fill (0,0) circle[radius=2mm];},node contents={}},
cc2/.style={circle,draw,inner sep=0pt,minimum size=0.75cm,path picture={
\fill (0,0) circle[radius=2mm];},node contents={}},scale=1.5]
\path (0,0) node[cc2];
\path (1,0) node[cc2];
\path (2,0) node[cc2];
\draw (0,0) -- (1,0);
\draw (1,0) -- (2,0);
\end{tikzpicture}} & $24$ & $1$ & $23$\\
    \hline
\resizebox{!}{0.23cm}{
\begin{tikzpicture}[line width=0.045cm,cc1/.style={minimum size=0.75cm,path picture={
\fill (0,0) circle[radius=2mm];},node contents={}},
cc2/.style={circle,draw,inner sep=0pt,minimum size=0.75cm,path picture={
\fill (0,0) circle[radius=2mm];},node contents={}},scale=1.5]
\path (0,0) node[cc2];
\path (1,0) node[cc2];
\path (2,0) node[cc1];
\draw (0,0) -- (1,0);
\draw (1,0) -- (2,0);
\end{tikzpicture}} & $6$ & $4$ & $20$ \\
    \hline\resizebox{!}{0.23cm}{
\begin{tikzpicture}[line width=0.045cm,cc1/.style={minimum size=0.75cm,path picture={
\fill (0,0) circle[radius=2mm];},node contents={}},
cc2/.style={circle,draw,inner sep=0pt,minimum size=0.75cm,path picture={
\fill (0,0) circle[radius=2mm];},node contents={}},scale=1.5]
\path (0,0) node[cc2];
\path (1,0) node[cc1];
\path (2,0) node[cc2];
\draw (0,0) -- (1,0);
\draw (1,0) -- (2,0);
\end{tikzpicture}} & $8$ & $3$ & $6$ \\
    \hline\resizebox{!}{0.23cm}{
\begin{tikzpicture}[line width=0.045cm,cc1/.style={minimum size=0.75cm,path picture={
\fill (0,0) circle[radius=2mm];},node contents={}},
cc2/.style={circle,draw,inner sep=0pt,minimum size=0.75cm,path picture={
\fill (0,0) circle[radius=2mm];},node contents={}},scale=1.5]
\path (0,0) node[cc2];
\path (1,0) node[cc1];
\path (2,0) node[cc1];
\draw (0,0) -- (1,0);
\draw (1,0) -- (2,0);
\end{tikzpicture}} & $4$ & $6$ & $12$ \\
    \hline
    \resizebox{!}{0.23cm}{
\begin{tikzpicture}[line width=0.045cm,cc1/.style={minimum size=0.75cm,path picture={
\fill (0,0) circle[radius=2mm];},node contents={}},
cc2/.style={circle,draw,inner sep=0pt,minimum size=0.75cm,path picture={
\fill (0,0) circle[radius=2mm];},node contents={}},scale=1.5]
\path (0,0) node[cc1];
\path (1,0) node[cc1];
\path (2,0) node[cc1];
\draw (0,0) -- (1,0);
\draw (1,0) -- (2,0);
\end{tikzpicture}} & $24$ & $1$ & $24$ \\
\hline
\end{tabular}
\quad
\begin{tabular}{|c|c|c|c|}
\hline
$I$ & $N_I$ & $R_I$ & $D_I$
\\
\hline
\resizebox{!}{0.23cm}{
\begin{tikzpicture}[line width=0.045cm,cc1/.style={minimum size=0.75cm,path picture={
\fill (0,0) circle[radius=2mm];},node contents={}},
cc2/.style={circle,draw,inner sep=0pt,minimum size=0.75cm,path picture={
\fill (0,0) circle[radius=2mm];},node contents={}},scale=1.5]
\path (0,0) node[cc2];
\path (1,0) node[cc2];
\path (2,0) node[cc2];
\path (3,0) node[cc2];
\draw (0,0) -- (1,0);
\draw (1,0) -- (2,0);
\draw (2,0) -- (3,0);
\end{tikzpicture}} & $120$ & $1$ & $119$\\
    \hline
    \resizebox{!}{0.23cm}{
\begin{tikzpicture}[line width=0.045cm,cc1/.style={minimum size=0.75cm,path picture={
\fill (0,0) circle[radius=2mm];},node contents={}},
cc2/.style={circle,draw,inner sep=0pt,minimum size=0.75cm,path picture={
\fill (0,0) circle[radius=2mm];},node contents={}},scale=1.5]
\path (0,0) node[cc2];
\path (1,0) node[cc2];
\path (2,0) node[cc2];
\path (3,0) node[cc1];
\draw (0,0) -- (1,0);
\draw (1,0) -- (2,0);
\draw (2,0) -- (3,0);
\end{tikzpicture}} & $24$ & $5$ & $115$ \\
    \hline
\resizebox{!}{0.23cm}{
\begin{tikzpicture}[line width=0.045cm,cc1/.style={minimum size=0.75cm,path picture={
\fill (0,0) circle[radius=2mm];},node contents={}},
cc2/.style={circle,draw,inner sep=0pt,minimum size=0.75cm,path picture={
\fill (0,0) circle[radius=2mm];},node contents={}},scale=1.5]
\path (0,0) node[cc2];
\path (1,0) node[cc2];
\path (2,0) node[cc1];
\path (3,0) node[cc2];
\draw (0,0) -- (1,0);
\draw (1,0) -- (2,0);
\draw (2,0) -- (3,0);
\end{tikzpicture}} & $12$ & $10$ & $50$ \\
    \hline
    \resizebox{!}{0.23cm}{
\begin{tikzpicture}[line width=0.045cm,cc1/.style={minimum size=0.75cm,path picture={
\fill (0,0) circle[radius=2mm];},node contents={}},
cc2/.style={circle,draw,inner sep=0pt,minimum size=0.75cm,path picture={
\fill (0,0) circle[radius=2mm];},node contents={}},scale=1.5]
\path (0,0) node[cc2];
\path (1,0) node[cc2];
\path (2,0) node[cc1];
\path (3,0) node[cc1];
\draw (0,0) -- (1,0);
\draw (1,0) -- (2,0);
\draw (2,0) -- (3,0);
\end{tikzpicture}} & $12$ & $10$ & $100$ \\
    \hline
    \resizebox{!}{0.23cm}{
\begin{tikzpicture}[line width=0.045cm,cc1/.style={minimum size=0.75cm,path picture={
\fill (0,0) circle[radius=2mm];},node contents={}},
cc2/.style={circle,draw,inner sep=0pt,minimum size=0.75cm,path picture={
\fill (0,0) circle[radius=2mm];},node contents={}},scale=1.5]
\path (0,0) node[cc2];
\path (1,0) node[cc1];
\path (2,0) node[cc2];
\path (3,0) node[cc1];
\draw (0,0) -- (1,0);
\draw (1,0) -- (2,0);
\draw (2,0) -- (3,0);
\end{tikzpicture}} & $8$ & $15$ & $30$ \\
    \hline
    \resizebox{!}{0.23cm}{
\begin{tikzpicture}[line width=0.045cm,cc1/.style={minimum size=0.75cm,path picture={
\fill (0,0) circle[radius=2mm];},node contents={}},
cc2/.style={circle,draw,inner sep=0pt,minimum size=0.75cm,path picture={
\fill (0,0) circle[radius=2mm];},node contents={}},scale=1.5]
\path (0,0) node[cc2];
\path (1,0) node[cc1];
\path (2,0) node[cc1];
\path (3,0) node[cc1];
\draw (0,0) -- (1,0);
\draw (1,0) -- (2,0);
\draw (2,0) -- (3,0);
\end{tikzpicture}} & $12$ & $10$ & $60$ \\
    \hline
    \resizebox{!}{0.23cm}{
\begin{tikzpicture}[line width=0.045cm,cc1/.style={minimum size=0.75cm,path picture={
\fill (0,0) circle[radius=2mm];},node contents={}},
cc2/.style={circle,draw,inner sep=0pt,minimum size=0.75cm,path picture={
\fill (0,0) circle[radius=2mm];},node contents={}},scale=1.5]
\path (0,0) node[cc1];
\path (1,0) node[cc1];
\path (2,0) node[cc1];
\path (3,0) node[cc1];
\draw (0,0) -- (1,0);
\draw (1,0) -- (2,0);
\draw (2,0) -- (3,0);
\end{tikzpicture}} & $120$ & $1$ & $120$ \\
    \hline
\end{tabular}

\begin{align*}
    \dim \mathcal{C}_{A_2}(v) & = 1 \cdot 5 + 2 \cdot 3 + 1 \cdot 6\\
    & = 20\\
    \dim \mathcal{C}_{A_3}(v) & = 1 \cdot 23 + 4 \cdot 20 + 3 \cdot 6 + 6 \cdot 12 + 1 \cdot 24\\
    & = 217\\
    \dim \mathcal{C}_{A_4}(v) & = 1 \cdot 119 + 5 \cdot 115 + 10 \cdot 50 + 10 \cdot 100 + 15 \cdot 30 + 10 \cdot 60 + 1 \cdot 120\\
    & = 3364
\end{align*}
    \caption{\label{fig:smallrank}The cases $n = 2, 3, 4$.}
\end{figure}

Going beyond these small rank computations, using a computer we were able to compute the dimension of $\mathcal{C}_{A_n}$ for all for $n \leq 43$, and we list the first $12$ in Figure \ref{fig:dimtable}.
\begin{figure}
    \centering
    \begin{tabular}{|c|c|}
    \hline
    $n$ & $\dim \mathcal{C}_{A_n}(v)$\\
    \hline
     $1$ & $3$\\ $2$ & $20$\\ $3$ & $217$\\ $4$ & $3364$\\ $5$ & $71098$\\ $6$ & $1960867$\\ $7$ & $67886033$\\ $8$ & $2871659468$\\ $9$ & $145498348666$\\ $10$ & $8683447971439$\\
     $11$ & $601843453126056$\\ $12$ & $47875219836485209$\\
     \hline
\end{tabular}
\caption{\label{fig:dimtable}A table of dimensions up to $n = 12$.}
\end{figure}

\clearpage

\bibliographystyle{alpha}
\bibliography{bibl}

\end{document}